\documentclass[10pt]{amsart}
\usepackage[normalem]{ulem}
\usepackage{amssymb, amsmath,xcolor}
\newtheorem{theorem}{Theorem}[section]
\newtheorem{proposition}{Proposition}[section]
\newtheorem{lemma}{Lemma}[section]
\newtheorem{corollary}{Corollary}[section]
\newtheorem{remark}{Remark}[section]
\newtheorem{example}{Example}[section]
\DeclareMathOperator{\ini}{in}

\DeclareMathOperator{\HF}{HF}

\title[Forcing the WLP for equigenerated monomial ideals]{Forcing the weak Lefschetz property for equigenerated monomial ideals}

\author{Nasrin Altafi}\address{Department of Mathematics and Statistics, Queen’s University, Kingston, Ontario, Canada and Department of Mathematics, KTH Royal Institute of Technology, Sweden}\email{nasrinar@kth.se} 

\author{Samuel Lundqvist}\address{Department of Mathematics, Stockholm University, Stockholm, Sweden}\email{samuel@math.su.se}


\begin{document}

\maketitle

\begin{abstract}

We classify the minimal number of generators of artinian equigenerated monomial ideals $I$
such that $\Bbbk[x_1,\ldots,x_n]/I$ is forced to have the weak Lefschetz property.
\end{abstract}

\tableofcontents

\tableofcontents
\section{Introduction} 

Let $A =\oplus A_i$ be a standard graded artinian algebra. A natural question to ask is whether there exists a form $f$ of degree $d$ in $A$ such that the multiplication by $f$ map has maximal rank, that is, that the map $\times f: A_i \to A_{i+d}$ is either surjective or injective for each $i$.

This question is tightly connected to the longstanding Fr\"oberg conjecture \cite{Fr}, which states that for generic forms $f_1, \ldots, f_m$, the algebra
$\mathbb{C}[x_1,\ldots,x_n]/(f_1,\ldots,f_{m})$ has this property. 

The case $m \leq n$ of the Fr\"oberg conjecture is classically known and follows from the fact that in this case generic forms constitute a regular sequence. One of the few instances were the conjecture is settled is the $m=n+1$-case, and this result is due to Stanley \cite{stanley} and independently by Watanabe \cite{Watanabe}. In fact, Stanley and Watanabe proved something stronger; that the multiplication by $(x_1+\cdots + x_n)^i$ map on the monomial complete intersection $\Bbbk[x_1,\ldots,x_n]/(x_1^{d_1},\ldots,x_n^{d_n})$ has maximal rank for each $i$ provided that the characteristic of the field $\Bbbk$ is zero.

In the nineties several researchers, inspired by the result by Stanley and Watanabe, began to systematically study the maximal rank question for various kinds of algebras, and in particular 
 when there exists a linear form $\ell$ in $A$ such that multiplication by $\ell$  has maximal rank for every degree, and stronger, when  there exists a linear form $\ell$ in $A$ such that multiplication by $\ell^i$  has maximal rank for every $i$.

These two properties were named the \emph{weak Lefschetz property (WLP)} and the \emph{strong Lefschetz property (SLP)} in a paper by Harima, Migliore, Nagel, and Watanabe \cite{unimodal}, motivated by the fact that Stanley used the hard Lefschetz theorem to prove his result, and have since then been adapted as the standard terminology by the community. 

The study of the Lefschetz properties has become a central area in commutative algebra and have provided interesting connections to fields such as algebraic geometry, combinatorics, and representation theory. See the overview paper \cite{atour} for an introduction and key references.

A well studied class with respect to these properties are the monomial algebras, see for instance  
\cite{AB,CN1,CN,MM,MNS} in characteristic zero, and \cite{BK,KV,LZ,LN} 
in characteristic $p$. While this class is interesting in its own right, it is also connected to the general case by Conca's  result \cite{conca} for $\Bbbk$ infinite;
if $\Bbbk[x_1,\ldots,x_n]/(\ini(I))$ has the WLP, then so does $\Bbbk[x_1,\ldots,x_n]/I$, where $\ini(I)$ denotes the initial ideal of $I$ with respect to a fixed monomial ordering.

A natural question is whether one can draw conclusions on the Lefschetz properties directly from the value of the Hilbert function of the algebra. We narrow the question by assuming $\Bbbk$ to be of characteristic zero.
A pioneering result in this direction is due to Migliore and Zanello \cite{MZ} who gave an elegant characterization of the Hilbert functions which force all artinian algebras to have the WLP.  

However, if one restricts the space of artinian algebras, the classification becomes more difficult. For the case of equigenerated monomial ideals, there are two major results with respect to the WLP.    
The first is due to Mezzetti and 
Mir\'o-Roig \cite{MM} who 
gave a sharp lower bound for the minimal number of generators of a monomial ideal $I$ of degree $d$ such that $\Bbbk[x_1,\ldots,x_n]/I$ 
fails the WLP by injectivity from degree $d-1$ to $d$. 
In the second, by the first author and Boij \cite{AB}, a corresponding result for the minimal number of generators of $I$ of degree $d$ such that $\Bbbk[x_1,\ldots,x_n]/I$ fails the WLP by surjectivity from degree $d-1$ to $d$ is established.

In this paper we continue on this path and arrive at a classification of the presence of the WLP in terms of the number of minimal generators of an equigenerated monomial ideal. Our main result is Theorem \ref{thm:main} which is stated when the number of variables is at least three. We remind the reader that all algebras of codimension 1 and 2 have the SLP \cite{unimodal}, and thus also the WLP.

\begin{theorem} \label{thm:main}
 Let $n \geq 3$ and $d \geq 2$ be integers and let $S=\Bbbk[x_1,\ldots,x_n]$, where $\Bbbk$ is a field of characteristic zero. Let
$$
\alpha(n,d) = \begin{cases}
4 & \text{if } n =3, d \equiv 3 \pmod 6\\
5  & \text{if } n =3, d \not\equiv 3 \pmod 6\\ 
n+2 & \text{if } n \geq 4 \text{ is even}, d = 2 \\
9  & \text{if } n = 5, d = 2\\
n+3  & \text{if } n \geq 7 \text{ is odd}, d= 2\\
6 & \text{if } n = 4, d=3\\
n+1  & \text{otherwise (if } n \geq 5, d \geq 3 \text{ or } n \geq 4, d \geq 4 \text{)}\\
\end{cases}
$$
and let 
$$\beta(n,d) =\binom{n+d-1}{d}-\begin{cases}
3(d-1) & \text{if } n=3 \text{ and } d\geq 3 \text{ is odd,} \\
3(d-1)+1 & \text{if } n=3 \text{ and } d\geq 2 \text{ is even,} \\
2d & \text{if } n\geq 4\text{ and } d\geq 2.
\end{cases}$$
Then for every $$\mu \in \Sigma_{n,d}:=
\begin{cases}
   [\alpha(n,d),\beta(n,d)] \setminus \{10\} = \{8,9,11,12,\ldots,17\}&  \text{ if } (n,d) = (6,2)\\
 [\alpha(n,d),\beta(n,d)] & \text{ otherwise}\\
 \end{cases}
 $$
  there exists an artinian ideal $I$ in $S$ minimally generated by $\mu$ elements of degree $d$ such that $S/I$ fails the WLP. 
 
 Moreover, the bounds are sharp, which means that if $\mu \notin \Sigma_{n,d}$, then 
for every artinian monomial ideal $I$ minimally generated by $\mu$ elements of 
degree $d$, the quotient $S/I$ has the WLP. Finally, the only case for which the interval 
$[\alpha(n,d),\beta(n,d)]$ is empty is when $(n,d) = (3,2)$. 
\end{theorem}

\begin{proof}
The proof is given in Section \ref{sec:seven}.
\end{proof}

We remark that the upper bound $\beta(n,d)$ equals to bound given in \cite{AB} for failing due to surjectivity from degree $d-1$ to $d$.  

Notice that, for large enough $n$ and $d$, the interval in question becomes 
$$\left[n+1,\dim_{\Bbbk} \Bbbk[x_1,\ldots,x_n]_d-2d\right],$$ 
and shows that algebras failing the WLP occurs for almost all values of the number of minimal generators. Notice also that for the border cases $n=3$ and $d=2$, the intervals shows periodic structures.

We find it surprising that there is a gap in the interval in the case $n=6$, $d=2$ where all artinian equigenerated monomial algebras with $10$ generators actually enjoys the WLP.

\addtocontents{toc}{\protect\setcounter{tocdepth}{1}}
\section{Preliminaries and an overview of the paper} \label{sec:prel}

Throughout this paper, let $\Bbbk$ be a field, which we in Sections 4-8 assume to be of characteristic zero.
Let $S=\Bbbk[x_1,\ldots,x_n]$ be the polynomial ring in $n$ variables equipped with the standard grading. For
a monomial ideal $I$, the quotient 
$A=S/I$ is a graded algebra, and we denote by $A_i$ the $i$'th graded piece of $A_i$. The \emph{Hilbert function} of $A$ is the function $i \mapsto \dim_{\Bbbk} A_i$ and we denote by $\HF(A,i)$ the value of this function in degree $i$. The \emph{Hilbert series} of $A$ is the power series $\sum_{i \geq 0} \HF(A,i) T^i$.

An algebra $A$ that is the quotient of a monomial ideal has the WLP if and only if multiplication by $x_1+\cdots+x_n$ has maximal rank in each degree, see \cite{MMN} for $\Bbbk$ infinite and \cite{LN} for $\Bbbk$ finite,  
so in the remaining part of this paper, we let $\ell = x_1+\cdots + x_n$ whenever $A$ has this property. We will abuse notation and sometimes regard $\ell$ as an element of a quotient $\Bbbk[x_1,\ldots,x_n]/I$, and sometimes as an element in the polynomial ring $\Bbbk[x_1,\ldots,x_n]$.

If $A$ is an artinian graded algebra failing the WLP, then there is an integer $i$ such that the map $ \times \ell: A_i \to A_{i+1}$ fails to have maximal rank. We say that \emph{$A$ fails the WLP in degree $i$}. If in addition we know that $\HF(A,i) \leq \HF(A,i+1)$, then failing the WLP in degree $i$ is equivalent to saying that the map $ \times \ell: A_i \to A_{i+1}$ fails injectivity. We say that \emph{$A$ fails the WLP by injectivity in degree $i$}. Likewise, if  $\HF(A,i) \geq \HF(A,i+1)$ and $A$ fails the WLP in degree $i$, we say that \emph{$A$ fails the WLP by surjectivity in degree $i$}.

We say that an artinian graded algebra $A$ has a 
\emph{isolated peak} if there is a positive integer $i$ such that $\HF(A,i-1) < \HF(A,i) > \HF(A,i+1)$.

The theory of \emph{Inverse System}, also known as \emph{Macaulay Inverse System} or \emph{Matlis duality},  will be used to prove some of the results in this paper.  The brief description is as follows,  for more details see \cite{Geramita} and \cite{IK}.  Let $R=\Bbbk [X_1,\dots ,X_n]$  be a graded $S$-module where $S$ acts on $R$ by contraction.  Over a field $\Bbbk$ of characteristic zero, we can, and will throughout the paper, let $S$ act by differentiation; i.e. for every $j$, and  every homogeneous polynomial $F\in R$, $x_j\circ F :=\partial F/ \partial X_j$.  For a homogeneous ideal $I\subset S$, the inverse system $I^{-1}\subset  R$ is defined as $I^{-1}:=\{g\in R\mid f\circ g=0, \hspace*{2mm}\text{for all}\hspace*{2mm} f\in I\}$.  We observe that when $I$ is a monomial ideal $(I^{-1})_d$ is generated by monomials in $R_d$ which correspond to monomials in $S_d$ but not in $I_d$. The Hilbert function of an artinian algebra $A=S/I$ can be computed by the inverse system module $I^{-1}$; $\HF(A,i)=\dim_{\Bbbk}(I^{-1})_i$ for every $i$. With this construction there exists a one-to-one correspondence between graded artinian algebras $S/I$ and finitely generated $S$-modules $M$ of $R$.  Using this correspondence,  we get that  the rank of multiplication map $\times \ell:(S/I)_{j-1}\rightarrow (S/I)_j$ for every linear form $\ell\in S_1$ and every $j$ is equal to the rank of differentiation map  $\circ \ell:(I^{-1})_{j}\rightarrow (I^{-1})_{j-1}$.

\subsection*{Outline of the techniques used in the proofs}\label{subsec:techniques}

To prove that an artinian algebra fails the WLP, there are, in principle, three different approaches.
\begin{enumerate}
    \item Show failure due to injectivity. 
    \item Show failure due to surjectivity.
    \item Show that the multiplication by $\ell$ map $A_i \to A_{i+1}$ is neither injective or surjective for some $i$.
\end{enumerate}
To show failure due to injectivity (surjectivity), one needs to show that for some $i$, the map $A_i \to A_{i+1}$ is not injective (surjective), and that this is \emph{unexpected}. 

To show unexpectedness, one needs knowledge of the difference of the values of the Hilbert function in degrees $i$ and $i+1.$
For failing due to injectivity, this difference has to be non-positive, and for failing due to surjectivity, the difference has to be non-negative. But determining the sign of the difference is a hard problem in general, even in the monomial case.

The third approach uses the very definition of failure of the WLP, and accordingly contains the least amount of information on $A$ of the three approaches.
For instance, failing due to injectivity in degree $i$ implies that the multiplication by $\ell$ map $A_i \to A_{i+1}$ is neither injective or surjective (the map is expected to be injective, that is, $\HF(A,i) \leq \HF(A,i+1)$, so if it fails to be injective, it also fails to be surjective). The same reasoning goes for failing due to surjectivity. 

A useful aspect of the third approach is that it is naturally independent of the calculation of the value of the Hilbert function. But we are not equipped with a lot of Hilbert free techniques for showing failure of injectivity and surjectivity at the same time, other than finding explicit elements both in the kernel and in the cokernel. And in some sense, this gives more qualitative information about the structure of the algebra. Thus, this constructive way of performing the third approach should not be expected to be an easier path than the first or the second approach.

This third approach appears for instance in a theorem by Migliore, Nagel, and Schenck \cite[Theorem 4.2]{MNS}, based on a lemma by Boij, Migliore, Mir\'o-Roig, Nagel, and Zanello \cite[Lemma 7.8]{memoir}. We use an adoption of 
the latter lemma 
to give a general way of constructing algebras failing the WLP in the spirit of the third approach. The result, Theorem \ref{thm:tensor}, states that if $A$ fails the WLP, and the Hilbert series of $B$ has an isolated peak, then $A \otimes B$ fails the WLP. From Theorem \ref{thm:tensor} we derive Corollary \ref{cor:main}, which will be our key result for our inductive arguments in Section \ref{sec:n+2} and Section \ref{sec:n+1}.

A huge part of the paper is devoted to establishing base cases for the induction. Since the induction is on the number of variables, we need a base case for each $d$, and
this turns out to be the most technically challenging part. Nevertheless, we make an effort to vary our proof methods. For instance, we succeed using the third approach for all proofs in Section \ref{sec:n+2}, and it is our hope that the reader will find these proofs interesting in its own right.

Tying together our inductive arguments with the aforementioned results by Mezzetti and Mir\'o-Roig, and the first author and Boij, gives us a proof of the existence of algebras failing the WLP for each $\mu \in \Sigma_{n,d}$.

To deal also with the sharpness part, we need to go in the other direction, proving that certain families of algebras have the WLP. Here the argument is more direct and we can in most cases lean on results available in the literature.

\subsection*{Organization of the paper}

In Section \ref{sec:construct} 
we state and prove Theorem \ref{thm:tensor} and its corollary, Corollary \ref{cor:main}, which is our main tool for construction algebras failing the WLP. In Section \ref{sec:n+2} we provide a large class of algebras failing the WLP; these are algebras with at least $n+2$ generators. 
In Section \ref{sec:n+1} we extend the class in Section \ref{sec:n+2} to also include $n+1$ generators;  the almost complete intersection case.
In Section \ref{sec:six} we deal with the sharpness part of our main theorem, proving the presence of the WLP. 
In Section \ref{sec:seven} we collect our results from the previous sections and give the proof of Theorem \ref{thm:main}. In Section \ref{sec:rmk} we conclude with a remark on Corollary \ref{cor:main}.

\addtocontents{toc}{\protect\setcounter{tocdepth}{5}}
\section{Constructing algebras failing the WLP} \label{sec:construct}
\setcounter{subsection}{1}
We recall the following lemma.
\begin{lemma} \label{lemma:memoir}
\cite[Lemma 7.8]{memoir}
Let $A = A' \otimes A''$ be a tensor product of two graded artinian $\Bbbk$-algebras $A'$ and $A''$. Let $\ell' \in A'$ and $\ell'' \in A''$  be linear elements, and set 
$\ell:=\ell' +\ell'' =\ell' \otimes 1+1 \otimes \ell'' \in A$. Then:
\begin{itemize}
\item If the multiplication maps $\times \ell': A'_{i-1} \to A'_i$ and $\times \ell'': A''_{j-1}
\to A''_{j}$ are both not surjective, then the multiplication map
$$ \times \ell : A_{i+j-1} \to A_{i+j}$$ is not surjective.

\item If the multiplication maps 
$\times \ell' : A'_i \to A'_{i+1}$
and
$\times \ell'' : A''_j \to A''_{j+1}$
are both not injective, then the multiplication map
$$ \times \ell : A_{i+j} \to A_{i+j+1}$$ is not injective.
\end{itemize}
\end{lemma}

As a direct consequence of Lemma  \ref{lemma:memoir} we obtain a method for constructing algebras that fail the WLP.

\begin{theorem} \label{thm:tensor}

Let $\Bbbk$ be a field.
Let $A'$ 
be a standard graded artinian algebra over $\Bbbk$ failing the WLP. Let $A''$
be a standard graded artinian algebra over $\Bbbk$ whose Hilbert series has at least one isolated peak. Then $A = A' \otimes A''$ fails the WLP.

Specifically, if $A'$ fails the WLP in degree $i$ and the Hilbert series of $A''$ has an isolated peak in degree $j$, then $A = A' \otimes A''$ fails the WLP in degree $i+j$.

\end{theorem}

\begin{proof}

Let $\ell$ be a general form in $A$. We can assume that 
$\ell = \ell' + \ell''$, where $\ell'$ is a general form in $A'$ and $\ell''$ is a general form in $A''$. Since $A'$ fails the WLP, there is an integer $i$ such that the map $\times \ell': A'_i \to A'_{i+1}$ is neither surjective nor injective. By the assumption on $A''$, the map 
$\times \ell'' : A''_{j-1} \to A''_j$ fails surjectivity, and the map $\times \ell'' : A''_{j} \to A''_{j+1}$ fails injectivity.

By Lemma \ref{lemma:memoir}, the map $A_{i+j} \to A_{i+j+1}$ fails both injectivity and surjectivity. Thus $A$ fails the WLP in degree $i+j$.

\end{proof}

For convenience of the reader, we also give a self contained proof of Theorem \ref{thm:tensor}.

\begin{proof}[Second proof of Theorem \ref{thm:tensor}]

We can write $A' = \Bbbk[x_1,\ldots,x_n]/I'$ for some $n$ and $I'$, and $A'' =  \Bbbk[x_1,\ldots,x_m]/I''$ for some $m$ and $I''$. As we have remarked in Section \ref{sec:prel}, if $A'$ fails the WLP by injectivity (surjectivity) in degree $i$, then the multiplication by $\ell'$ map $A'_i \to A'_{i+1}$ is not surjective (injective) either. Thus, in both cases there is an element $f$ in the kernel of the map $\times \ell': A'_i \to A'_{i+1}$, and an element $F$ of degree $i+1$ in the kernel of the contraction map $\circ \ell':  (I')^{-1}_{i+1} \to (I')^{-1}_{i}$.

The assumption on $A''$ to have an isolated peak in degree $j$ provides us with an element $g$ of degree $j$ in the kernel of $\times \ell'': A''_j \to A''_{j+1}$, and an element $G$ of the same degree degree $j$ in the kernel of the contraction map $\circ \ell'':  (I'')^{-1}_{j} \to (I'')^{-1}_{j-1}$.

It is now straightforward  to check that $f g$ of degree $i+j$ is in the kernel of the map $\times (\ell'+\ell'')$ from $A_{i+j} \to A_{i+j+1}$, and that $FG$ of degree $i+j+1$ is in the kernel of the contraction map $\circ (\ell'+\ell''):  (I)^{-1}_{i+j+1} \to (I)^{-1}_{i+j}$.

Thus 
the multiplication map from $A_{i+j} \to A_{i+j+1}$ is neither surjective or injective.
\end{proof}

\begin{remark} \label{rmk:cor}
We can use the second proof of Theorem \ref{thm:tensor} to obtain even more information about the dimension in degree $i+j+1$ of $A$. Indeed, 
if $f_1, \ldots, f_r$ are in the kernel of the $A'$-map, and if $g_1, \ldots, g_s$ are in the kernel of the $A''$-map, then each element in $\{f_i g_j\}$ is in the kernel of the $A$-map, and they are linearly independent. Similarly for the inverse system; if $F_1, \ldots, F_{r'}$ of degrees $d$ are in the inverse system of $I' + (\ell)'$, and $G_1,\ldots,G_{s'}$ of degrees $j$ are in the inverse system of $I'' + (\ell'')$, then every element in the linearly independent set $\{F_i G_j\}$  is in the inverse system of 
$I'+I''+ (\ell' + \ell'')$. In particular, the kernel has dimension at least $rs$ and the cokernel has dimension at least $r's'$.
\end{remark}

We will use Theorem  \ref{thm:tensor} for a special class of artinian algebras having isolated peaks; equigenerated monomial complete intersections in two variables. To simplify the presentation of the paper we state the following corollary of Theorem \ref{thm:tensor} for this class. 

\begin{corollary} \label{cor:main}
Let $\Bbbk$ be a field and let $A=\Bbbk[x_1,\ldots,x_n]/I$ be an artinian graded algebra failing the WLP in degree $i$. Let $B=\Bbbk[y_1,y_2]/(y_1^d,y_2^d)$. Then $R = A \otimes B$ fails the WLP in degree $i + d - 1$. 

Moreover, let $F$ be an element of degree $i+1$ in the inverse system of $I+(x_1+\cdots + x_n)$, and let $f$ be an element in the kernel of the multiplication by $x_1+\cdots + x_n$ map from $A_i$ to $A_{i+1}$. Then $$F \cdot (Y_1-Y_2)^{d-1}$$ of degree $i+d$ is in the inverse system of $I + (y_1^d,y_2^d) + (x_1+\cdots+x_n+y_1+y_2)$, and $$f \cdot (y_1^d + (-1)^{d-1} y_2^d)/(y_1+y_2)$$ is in the kernel of the multiplication by $x_1+\cdots+x_n+y_1+y_2$ map $R_{i+d-1} \to R_{i+d}$.

\end{corollary}
\begin{proof}
The Hilbert series of $B$ has an isolated peak in degree $d-1$. Hence we can apply Theorem \ref{thm:tensor} and conclude that $R$ fails the WLP in degree $i + d-1$.

For the explicit description, it is easy to see that $(Y_1-Y_2)^{d-1}$ 
is in the inverse system of $(y_1^d,y_2^d,y_1+y_2)$, and that $
$$(y_1+y_2) \cdot ((y_1^d + (-1)^{d-1} y_2^d)/(y_1+y_2)) = 0$ in $B$. Thus we can apply the same argument as in Remark \ref{rmk:cor} to draw the conclusion. 
\end{proof}

\begin{remark} \label{rmk:distinct}
If we drop the assumption on $B$ having an isolated peak in Theorem \ref{thm:tensor} the corresponding tensor algebra does not necessarily fail the WLP. Indeed, consider the Togliatti system defined by $I=(x_1^3,x_2^3,x_3^3,x_1 x_2 x_3) \subset \Bbbk[x_1,x_2,x_3]$. The quotient $A$ is known not to have the WLP, and this was first reported in the literature by Brenner and Kaid \cite{BK}. Let $B= \Bbbk[y]/(y^3)$, an algebra whose Hilbert series has a flat top. A calculation now reveals that $A \otimes B$ actually \emph{has} the WLP.

\end{remark}
As mentioned in the introduction, a result in the same spirit of Theorem \ref{thm:tensor}, also based on \cite[Lemma 7.8]{memoir}, is given by Migliore, Nagel, and Schenck \cite[Theorem 4.8] {MNS}. We will state this result in Section \ref{sec:n+2} and use it in the proof of Proposition \ref{prop:n,n+2,2}. 

Over fields of characteristic zero another construction of algebras failing the WLP is given by Babson and Nevo \cite[Lemma 4.3]{BN}. They inductively construct homology spheres such that their associated Stanley-Reisner rings fail the WLP. They show that the Stanley-Reisner ring of a join of a homology sphere failing the WLP with a simplex is a homology sphere which fails the WLP.

Our next corollary concerns the open question on whether all  artinian complete intersections in codimension $\geq 4$ have the WLP.
\begin{corollary} \label{cor:ci}
Suppose that for some positive integer $N$ it holds that all artinian complete intersections in characteristic zero in $n \geq N$ variables have the WLP. Then all artinian complete intersections have the WLP.
\end{corollary}
\begin{proof}
Suppose the contrary holds. Then there is an artinian complete intersection algebra $A$ with $n < N$ generators failing the WLP. Let $k$ be such that $n+2k>N$ and let $d \geq 2$. By repeated use of Corollary \ref{cor:main}, we have that $A \otimes \Bbbk[y_1,\ldots,y_{2k}]/(y_1^{d},\ldots,y_{2k}^{d})$ fails the WLP. But $A \otimes \Bbbk[y_1,\ldots,y_{2k}]/(y_1^{d},\ldots,y_{2k}^{d})$ is a complete intersection, contradicting the assumption.
\end{proof}

Having established Theorem \ref{thm:tensor} and Corollary \ref{cor:main}, we will in the rest of the paper work over 
a field $\Bbbk$ of characteristic zero. 

We finish the section by recalling a result on the structure of the Hilbert series of a complete intersection.

\begin{lemma}\label{lemma:hfci} Special case of \cite[Theorem 1]{rrr}.
The value of the Hilbert function of an artinian complete intersection in $n \geq 2$ variables generated in degree $d\geq 2$ is strictly increasing in the interval $[0, \lfloor n(d-1)/2\rfloor]$, constant in the interval $[\lfloor n(d-1)/2 \rfloor,\lceil n(d-1)/2 \rceil] $, and then strictly decreasing in the interval  $[\lceil n(d-1)/2 \rceil,n(d-1)]$. Moreover, the series is symmetric around  $n(d-1)/2$.
\end{lemma}

\section{Ideals with $\mu \geq n+2$ minimal generators and the failure of the WLP} \label{sec:n+2}

We recall from Theorem \ref{thm:main} that 
$$\beta(n,d) =\HF(S,d)-\begin{cases}
3(d-1) & \text{if } n=3 \text{ and } d\geq 3 \text{ odd,} \\
3(d-1)+1 & \text{if } n=3 \text{ and } d\geq 2 \text{ even,}\\
2d & \text{ if } n\geq 4\text{ and } d\geq 2.
\end{cases}$$

We will show that for $n \geq 3$, $d \geq 2$, and $\mu \in [n+2,\beta(n,d)]$, there exists an artinian monomial ideal $I$ minimally generated in degree $d$ by $\mu$ elements such 
 that $S/I$ fails the WLP. We remark that this interval is empty for $n=3, d=2$.
 
We will proceed by breaking up the interval $[n+2,\beta(n,d)]$ in three parts; 
\begin{align}
    [n+2, 2n-2], \\
    \left[2n-1,\HF(S,d)-\HF(S,d-1)\right], \\
   \left[\HF(S,d)-\HF(S,d-1), \beta(n,d) \right].
\end{align}

For the third interval we will rely on the result by the first author and Boij \cite{AB} on how to construct monomial ideals failing the WLP by surjectivity, 
and in Corollary \ref{cor:surj}, we will use this result to prove that for every $\mu\in \left[\HF(S,d)-\HF(S,d-1), \beta(n,d) \right]$ there is an ideal $I$ minimally generated by $\mu$ elements of degree $d$ such that $S/I$ fails the WLP.

Mezzetti, Mir\'o-Roig, and Ottaviani in \cite{MMO} described a relation between the existence of graded artinian ideals $I\subset S$ generated in degree $d$ such that $S/I$ fails the WLP and the existence of projective varieties satisfying at least one Laplace equation of order $d-1$.   Togliatti  in \cite{Togliatti1} and \cite{Togliatti2}  gave a classification of rational surfaces paramaterized by cubics and 	satisfying at least one Laplace equation of order 2.    In honour of Togliatti and the relation described in \cite{MMO}, the authors define a \emph{Togliatti system} to be an artinian ideal $I\subset S$ generated by homogeneous forms of degree $d$ such that $S/I$ fails the WLP in degree $d-1$ by injectivity.  Togliatti's classification shows that the only Togliatti system 
generated in degree 3 such that $S/I$ fails the WLP in degree 2 is $I=(x_1^3,x_2^3,x_3^3,x_1x_2x_3)$, an ideal which we have already encountered in Remark \ref{rmk:distinct}.

Mezzetti and Mir\'o-Roig \cite{MM} showed how to construct monomial Togliatti systems. More precisely, they construct monomial ideals with $2n-1$ minimal generators of degree $d$ such that $S/I$ fails the WLP in degree $d-1$ by injectivity. In Corollary \ref{cor:inj} we will use this result to deal with the second interval above and prove that for every $\mu\in \left[2n-1,\HF(S,d)-\HF(S,d-1)\right]$ there is an algebra whose corresponding ideal is minimally generated by $\mu$ generators of degree $d$ failing the WLP.

To handle the first interval we will begin by showing how to construct ideals with $2n-2$ generators failing the WLP, and then use Corollary \ref{cor:main} to extend the result. 

In the following section, we handle the second and the third mentioned intervals together. The first interval requires some more technical consideration and will be treated separately.

\subsection{The case $\mu \in  \left[2n-1,\beta(n,d) \right]$}
\noindent

 We first treat the case where $\mu\in \left[\HF(S,d)-\HF(S,d-1), \beta(n,d)\right]$.  This is in fact a consequence of the following result which is a reformulation of \cite[Theorems 4.3, and 5.2]{AB}. 

\begin{theorem}\label{thm:ALthm}
Let $n\geq 3,d\geq 2$  and $S=\Bbbk[x_1,\dots ,x_n]$, where $\Bbbk$ is a field of characteristic zero.  

\begin{itemize}
\item Except for $(n,d)=(3,2)$, there exists an artinian monomial ideal $I$ minimally generated by $\beta(n,d)$ elements of degree $d$ such that $S/I$ fails WLP by failing surjectivity in degree $d-1$.  Moreover,  the minimal set of generators of $I$ is given by
$$\mathcal{G}(I)= \{m\in S_d\mid m\circ f = 0\}$$ 
where 
$$f = \begin{cases}
(X_1-X_2)(X_1-X_3)(X_2-X_3)^{d-2}&  \text{if } n=3 \text{ and } d\geq 3,\\
(X_1-X_2)(X_3-X_4)^{d-1}& \text{if } n\geq 4 \text{ and } d\geq 2.\\
\end{cases}$$
\item For every artinian monomial ideal $I$ minimally generated by $\mu$ elements of degree $d$ such that $ \mu\in [\beta(n,d)+1,\HF(S,d)]$ the quotient algebra $S/I$ has the WLP.
\end{itemize}

\end{theorem}

\begin{corollary}\label{cor:surj}
Let $n \geq 3, d \geq 2$.
Then for every integer $$\mu\in \left[\HF(S,d)-\HF(S,d-1), \beta(n,d) \right],$$ there exists an artinian monomial ideal $I$ minimally generated by $\mu$ elements of degree $d$ such that $S/I$ fails the WLP by surjectivity in degree $d-1$. Moreover, the set $\left[\HF(S,d)-\HF(S,d-1), \beta(n,d) \right]$ is non-empty except for the case $n=3, d= 2.$
\end{corollary}
\begin{proof}
We first observe that the interval is non-empty except for $n=3,d=2$. For every $\mu$ in the above interval and every monomial ideal $I$ that is minimally generated by $\mu$ elements of degree $d$ we have that $\HF(S/I,d-1)\geq  \HF(S/I,d)$. When $(n,d) \neq (3,2)$, Theorem \ref{thm:ALthm} provides that there exists an artinian monomial ideal $I$ minimally generated by $\beta(n,d)$ elements of degree $d$ such that $S/I$ fails the WLP by surjectivity in degree $d-1$. We now observe that for every monomial ideal such that their minimal set of generators is included in the minimal set of generators of $I$ we still have that the multiplication map by $\ell$ in degree $d-1$ fails surjectivity. This finishes the proof.
\end{proof}
For the remaining cases we use the following result that is the summary of \cite[Theorem 3.9]{MM} and the discussion preceding \cite[Proposition 3.4]{MM}.
\begin{theorem}\label{thm:MMthm}
Let $n\geq 3,d\geq 2$ and $S=\Bbbk[x_1,\dots ,x_n]$, where $\Bbbk$ is a field of characteristic zero.  Let 
$$\nu(n,d) =\begin{cases}
2n-1 & \text{if } n\geq 3, d\geq 2  \hspace*{1mm}\text{except for} \hspace*{1mm} (n,d)\in \{(4,2),(3,3),(3,2)\},\\
4& \text{if } (n,d)=(3,3), \\
6 & \text{if } (n,d)=(4,2).
\end{cases}$$
\begin{itemize}
\item Then there exists an artinian monomial ideal $I$ minimally generated by $\nu(n,d)$ elements of degree $d$ such that $S/I$ fails WLP by failing injectivity in degree $d-1$. Moreover,  for every $(n,d), n\ge 3, d\ge 2$ except when  $(n,d)\in \{(4,2),(3,3),(3,2)\}$ this ideal can be chosen as $$(x_1^d,\ldots,x_n^d,x_1x_n^{d-1},  \ldots, x_{n-1} x_n^{d-1}).$$ 
For $(n,d)=(3,2)$ the WLP is forced, for $(n,d)=(4,2)$, $I$ can be chosen as 
$(x_1^2,\ldots,x_4^2,x_1x_2, x_3x_4)$, and for  $(n,d)=(3,3)$ we can let $I$ be the Togliatti system $I=(x_1^3,x_2^3,x_3^3,x_1x_2x_3)$.
\item Furthermore, for every artinian monomial ideal $I$ minimally generated by $\mu$ elements of degree $d$ such that $\mu\in \left[n,\nu(n,d)-1\right]$, the multiplication  by $\ell$ map on $S/I$ is injective in degree $d-1$.
\end{itemize}
\end{theorem}

\begin{corollary}\label{cor:inj}
Let $n\geq 3, d\geq 2$. Then, for every integer $$\mu \in\left[2n-1, \HF(S,d)-\HF(S,d-1)\right],$$ there exists an artinian monomial ideal $I$ minimally generated by $\mu$ elements of degree $d$ such that $S/I$ fails the WLP by injectivity in degree $d-1$. Moreover, the set $\left[2n-1, \HF(S,d)-\HF(S,d-1)\right]$ is non-empty except for the cases $(n,d) \in \{(4,2),(3,3),(3,2)\}$.
\end{corollary}

\begin{proof}
Notice that for every monomial ideal $I$ generated in degree $d$ such that the number of minimal generators $\mu$ is in the interval $\left[2n-1, \HF(S,d)-\HF(S,d-1)\right]$ we have that $\HF(S/I,d-1)\leq  \HF(S/I,d)$. 
Thus, the WLP fails in degree $d-1$ if and only if the injectivity fails from degree $d-1$ to $d$.

Theorem \ref{thm:MMthm} shows that when $\mu = 2n-1$ the algebra $S/I$, where $$I = (x_1^d,\ldots,x_n^d,x_1x_n^{d-1},  \ldots, x_{n-1} x_n^{d-1}),$$ fails the WLP by injectivity in degree $d-1$, and one can easily see that $ \ell x_n^{d-1} = 0$ in $S/I$.
Now let $J$ be an ideal minimally generated by $\mu$ elements, $\mu \in [2n,\HF(S,d)-\HF(S,d-1)],$ such that $I \subset J$. From the inclusion property, we have  and $\ell x_n^{d-1}=0$ in $S/J$. Thus $S/J$ fails the WLP by injectivity in degree $d-1$. 

We now turn to the last part. It is straightforward to check that the set is empty when $(n,d) \in \{(4,2),(3,3),(3,2)\}$. By Theorem \ref{thm:MMthm}, there is an ideal with $2n-1$ minimal generators failing the WLP by injectivity when $n \geq 4$ and $d \geq 3$, so the set $\left[2n-1, \HF(S,d)-\HF(S,d-1)\right]$ is non-empty in this case. For $n = 3$ and $d=4$ a calculation shows that $2n-1 =  \HF(S,d)-\HF(S,d-1)$, and since the difference $\HF(S,d)-\HF(S,d-1)$ increases when $d$ increases, we get that $\left[2n-1, \HF(S,d)-\HF(S,d-1)\right]$ is non-empty also for $n = 3$ and $d > 4$.
\end{proof}

\subsection{The case $\mu \in  \left[n+2, 2n-2 \right]$}
\noindent

In the following three lemmas, we show the existence of algebras $S/I$ failing WLP where $I$ is minimally generated by $2n-2$ monomials of degree $d$. 

\begin{lemma} \label{lemma:2n-2} 

Let $n$ and $d$ be integers fulfilling either $n \geq 5, d \geq 4$ or $n\geq 6, d\geq 3$, or $n \geq 7, d \geq 2$.   Then there exists an artinian monomial ideal $I$ minimally generated by $2n-2$ elements of degree $d$ such that $S/I$ fails the WLP 
%by injectivity 
in degree $2d-2$.
\end{lemma}

\begin{proof}
Let $I=(x_1^d,\dots ,x_n^d,x_1^{d-1}x_2,\dots ,x_1^{d-1}x_{n-1}).$ 

First note that for $f=x_1^{d-1}x_n^{d-1}\in (S/I)_{2d-2}$ we have that $\ell f =0$. Thus the map from degree $2d-2$ to degree $2d-1$ is not injective. 

We next want to show that the map is not surjective either. Thus we need to find an element in the kernel of the differential map from degree $2d-1$ to degree $2d-2$. When $n=5$ and $d \geq 4$, we see that the element $(X_1-X_2)  (X_1-X_3)  (X_2-X_3)^{d-2}  (X_4-X_5)^{d-1}$ is in the kernel. For $n \geq 6$ and $d \geq 3$ we can choose the element $(X_1-X_6) (X_2-X_3)^{d-1} (X_4-X_5)^{d-1}$, while if  $n \geq 7$ and $d \geq 2$, the element 
$(X_2-X_3)(X_4-X_5)^{d-1} (X_6-X_7)^{d-1}$
serves our purpose. 
\end{proof}

\begin{lemma}\label{lemma:2n-2,n=4}
Let $I=(x_1^d,\dots ,x_4^d,x_1^{d-1}x_2,x_{3}^{d-1}x_4)$, $d\geq 2$. Then $S/I$ fails the WLP 
%by surjectivity 
in degree $2d-3$.
\end{lemma}

\begin{proof}
We first observe that polynomial $F = (X_1-X_2)^{d-1}(X_3-X_4)^{d-1}\in (I^{-1})_{2d-2}$ is in the kernel of the differentiation map $\circ \ell : (I^{-1})_{2d-2}\rightarrow (I^{-1})_{2d-3}$. This shows that surjectivity fails in degree $2d-3$.

To show that injectivity fails in the same degree, consider first the univariate equality
$$(1+y) \cdot \left( \sum_{i=0}^{d-2} (-1)^{d-2-i}(i+1) y^i \right) = 
\sum_{i=0}^{d-2} (-1)^{d-2-i} y^i + (d-1) y^{d-1}.$$
Since $$ \sum_{i=0}^{d-2} (-1)^{d-2-i} y^i= \frac{(-1)^d + y^{d-1}}{1+y},$$
it follows that 
\begin{align*} 
\begin{split}
(1+y)^2 \cdot \left(\sum_{i=0}^{d-2} (-1)^{d-2-i}(i+1) y^i \right) =\\
(-1)^d + d y^{d-1} + (d-1) y^{d}.
\end{split}
\end{align*}
Homogenizing the above equality, we obtain 
\begin{align} \label{eq:uni}
\begin{split}
(z+y)^2 \cdot \left(\sum_{i=0}^{d-2} (-z)^{d-2-i}(i+1) y^i \right) =\\
(-z)^d + d z y^{d-1} + (d-1) y^{d}.
\end{split}
\end{align}
Let now $$h_1 = \sum_{i=0}^{d-2} (-1)^{d-2-i}(i+1) x_1^i x_2^{d-2-i} \text{ and }  h_2 = \sum_{i=0}^{d-2} (-1)^{d-2-i}(i+1) x_3^i x_4^{d-2-i}$$ be elements in $S$.

Using that $$\ell \cdot (x_1 + x_2 - x_3 - x_4) = (x_1+x_2)^2 - (x_3+x_4)^2,$$ and Equality 
(\ref{eq:uni}), 
we get that
$$\ell \cdot (x_1 + x_2 - (x_3 + x_4)) \cdot h_1 \cdot h_2 = (x_1+x_2)^2 h_1 h_2 - (x_3+x_4)^2 h_1 h_2 \in I.$$ 

The element $(x_1 + x_2 - (x_3 + x_4)) \cdot h_1 \cdot h_2$ is of degree $2d-3$, and 
since the monomial $x_1^{d-1} x_3^{d-1}$ occurs 
with a non-zero coefficient, it is a non-trivial element in the quotient.

This shows that the map from degree $2d-3$ to degree $2d-2$ is not injective. Thus $S/I$ fails the WLP in degree $2d-3$.

\end{proof}

We now treat two special cases for $n = 5$ and $d=3$.
\begin{lemma} \label{lemma:n=5,78}
Let $$I = (x_1^3,\dots , x_5^3, x_1^2x_2, x_1x_2^2) \text{ and let }
J = (x_1^3,\dots , x_5^3, x_1^2x_2, x_1x_2^2,x_3x_4x_5).$$ Then
$S/I$ and $S/J$ both fail the WLP %by surjectivity 
in degree $4$.
\end{lemma}
\begin{proof}

 We begin by noticing that $(X_3-X_4)(X_3-X_5)(X_4-X_5)(X_1-X_2)^2$ is a non-trivial element in the kernel of the differentiation map $\circ (x_1+\cdots +x_5):(I^{-1})_5\rightarrow (I^{-1})_4$.
We next see that the non-trivial element $$f=\frac{(x_4^3+x_5^3)}{(x_4+x_5)} \cdot 
\frac{ (x_1+x_2)^3 + x_3^3}{(x_1+x_2) + x_3}$$
is in the kernel of multiplication by $\ell$ map from degree $2$ to degree $3$ of $S/I$. Indeed, 

$$\ell f = (x_4+x_5) f + 
(x_1+x_2+x_3) f = 0 + 0 = 0.$$
Thus the WLP fails in degree $4$.  

The second algebra can be written as the tensor product of our Togliatti system $\Bbbk[x_3,x_4,x_5]/(x_3^3,x_4^3,x_5^3,x_3 x_4 x_5 )$ and the isolated peak in degree 2 algebra $\Bbbk[x_1,x_2]/(x_1^3,x_2^3,x_1^2 x_2, x_1 x_2^2)$, so by Theorem \ref{thm:tensor}, the algebra fails the WLP in degree 4.

\end{proof}

We conclude the section by extending the interval of algebras failing the WLP given in Corollary \ref{cor:inj}. In Proposition \ref{prop:n+2} we deal with the case $d \geq 3$, while Proposition  \ref{prop:d=2} is devoted to the quadratic case. 

 The proofs of Proposition \ref{prop:n+2} and Proposition  \ref{prop:d=2} will be on induction on the number of variables. To keep track of the number of variables in the polynomial ring $S$ when we refer to the value of the Hilbert function, we let $$\Delta(n,d) := \HF(S,d) - \HF(S,d-1).$$ Notice that $$\Delta(n,d) = \binom{n+d-1}{d} - \binom{n+d-2}{d-1} = \binom{n+d-2}{d}.$$
 
\begin{proposition} \label{prop:n+2}
Let $n \geq 4$ and $d \geq 3$. Then for every $$\mu\in \left[ n+2, \Delta(n,d)\right]$$ there exists an artinian monomial ideal $I\subset S$ generated by $\mu$ elements of degree $d$ such that $S/I$ fails the WLP.
\end{proposition}
\begin{proof}

We will use induction on the number of variables. Since we will increase $n$ by two in the induction step, we need two different base cases; one for $n=4$ and one for $n=5$.

Let $n=4$. By Lemma \ref{lemma:2n-2,n=4}, the statement holds for $\mu=6$, and by Corollary \ref{cor:inj}, the statement holds for $\mu \in [7,\Delta(4,d)].$

We now consider the case $n = 5.$ Let $d=3$. By Lemma \ref{lemma:n=5,78} the statement holds for $\mu \in [7,8]$. By Corollary  \ref{cor:inj}, the statement holds for $\mu \in [9,\Delta(5,d)].$ 
Consider next the case $d \geq 4$. By Lemma \ref{lemma:2n-2}, the statement holds for $\mu = 8$. By Corollary \ref{cor:inj}, the statement holds for $\mu \in [9,\Delta(5,d)].$ Thus we are left with the case $\mu = 7$. By Corollary \ref{cor:inj}, there is a codimension 3 algebra with $5$ generators failing the WLP. By using Corollary \ref{cor:main} we obtain a codimension $5$ algebra  with $7$ generators failing the WLP.

Having settled the base cases, suppose that the statement is true for some $(n,d)$, with $n \geq 4$ and  $d \geq 3$, so 
that for every $\mu \in \left[n+2, \Delta(n,d) \right]$ there is a monomial algebra in codimension $n$ failing the WLP. 

By applying Corollary \ref{cor:main} we get that for every $\mu \in \left[n+4, \Delta(n,d)+2 \right]$ there is an artinian monomial ideal in codimension $n+2$ generated by $\mu$ elements such that the quotient fails the WLP. 

We need to fill up the rest of the interval, more precisely, we need to show that for each $$
\mu \in \left[\Delta(n,d)+3, \Delta(n+2,d)] \right]
$$
there is an ideal in codimension $n+2$ generated by $\mu$ elements such that the quotient fails the WLP.

It is enough to show that \begin{equation} \label{eq:ineqprop1}
2(n+2)-1 \leq \Delta(n,d)+3
\end{equation}
since we in this case can use Corollary \ref{cor:inj} to draw our conclusion. 

The inequality (\ref{eq:ineqprop1}) is equivalent to 
$$
2n \leq \binom{n+d-2}{d},
$$
and for fixed $n$, the function $d \mapsto \binom{n+d-2}{d}$ is strictly increasing, so if we can show the inequality for one value of $d$, then it holds for all other larger values as well.
When $d=3$ the equality becomes $2n \leq \binom{n+1}{3}$ which holds for $n \geq 4$, which we have assumed. This concludes the proof.

\end{proof}
We will now treat the case $d=2$, were we will make use a result from the aforementioned paper by  Migliore, Nagel, and Schenck, together with some Diophantine arguments for the case $n$ even. When $n$ is odd, we will instead rely on Corollary \ref{cor:main}.

\begin{theorem} \label{thm:MNS}
\cite[Theorem 4.8.]{MNS}

If $char (\Bbbk) \neq 2$, and
$$
C = \bigotimes_{i=1}^m 
Sym(V_i)/V_i^2 \text{ with } \dim V_i = d_i$$
then $C$ has the WLP if and only if one of the following holds:
\begin{enumerate}
    \item $d_2,\ldots,d_m$ are $1$,
    \item $d_3,\ldots,d_m$ are $1$, and $m$ is odd.
\end{enumerate}

\end{theorem}

\begin{proposition} \label{prop:d=2}
Let $d=2$, and let $n \geq 4$. 

Let
$$\Omega_n = \begin{cases}
[6,\Delta(4,2)] & \text {if }n= 4,\\
[9,\Delta(5,2)] & \text {if }n= 5,\\
[8,\Delta(6,2)]  \setminus \{10\} & \text {if } n= 6,\\
[n+3,\Delta(n,2)] & \text { if } n \text{ is odd and } n \geq 7,\\
[n+2,\Delta(n,2)] & \text { if } n \text{ is even and } n \geq 8.\\
\end{cases}$$

Then, for $n \geq 4$ and every $\mu \in  \Omega_n$, there is an ideal in $S$ minimally generated by $\mu$ elements such that the corresponding quotient algebra fails the WLP. 
\end{proposition}

\begin{proof}

The proof consists of two parts.

{\bf The case $n$ even.} 

By Corollary \ref{cor:inj}, it is enough to prove the statement when $\Delta(n,2)$ is replaced by $2n-1$. Let $\Omega_n'$ denote the subset of $\Omega_n$ with respect to this restriction.

We consider now a a special case of Theorem \ref{thm:MNS}; when $d_i \leq 3$ in Theorem \ref{thm:MNS}. Let $a$ denote the number of $d_i$ which are equal to three, 
let $b$ denote the number of $d_i$ which are equal to two, and let $c$ denote the number of $d_i$ which are equal to one. Since the sum of the $d_i$ should equal $n$, we have $$n=3a+2b+c,$$ and this implies the inequality \begin{equation} \label{eq:prop2ineq}
    n-3a-2b \geq 0.
\end{equation}

Suppose that $$\begin{cases} \label{eq:prop2ineq2}
   a + b \geq 2 & \text{if the number of $d_i$'s is even} \\
   a + b \geq 3 & \text{if the number of $d_i$'s is odd}.
\end{cases}$$ Then we can let the remaining $d_i$'s by equal to one, and use Theorem \ref{thm:MNS} to conclude that there is an algebra with 
 $$6a + 3b + c = 6a + 3b + (n-3a-2b) = n + 3 a + b$$ generators failing the WLP.
Explicitly, this algebra is given by 
$$\underbrace{A \otimes \cdots \otimes A}_{a \text{ times}} \otimes \underbrace{B \otimes \cdots \otimes B}_{b \text{ times}}  \otimes
\underbrace{C \otimes \cdots \otimes C}_{c \text{ times}},$$
where 
$$A = \Bbbk[x_{1},x_{2},x_{3}]/
(x_{1},x_{2},x_{3})^2, \,
B = \Bbbk[y_{1}, y_{2}]/
(y_{1}, y_{2})^2, \, \text{ and  } \,
C = \Bbbk[z_1]/(z_1^2).$$

Thus we need to show that for each $\mu \in \Omega'_n$ we can write 
\begin{equation} \label{eq:prop2eq}
\mu = n + 3 a + b
\end{equation}
for $a,b$ respecting the inequalities given above.

If  $2 \leq \mu-n \leq n/2$ we can let $b = \mu-n$ and $a = 0$, a choice which will satisfy (\ref{eq:prop2ineq}) and (\ref{eq:prop2eq}). The number of $d_i$ is even, so we need to check that $a+b \geq 2$, which is clear by the assumption.

If instead $n/2 + 1 \leq \mu-n   \leq n-2$, we have 
$\mu-n = n/2 + i \leq n-2$ for some $i\geq 1.$ Let $b = n/2 - 2i$ and $a = i$. We have $n+3a+b=n+3i+n/2-2i = n + n/2 + i = \mu $, and 
$n - 3a - 2b= n-3i-n+4i=i\geq 0$, so this choice respects (\ref{eq:prop2eq}) and (\ref{eq:prop2ineq}). It rests to show that inequality concerning the components is satisfied. Notice that the number of components has the same parity as $b$. 

Suppose first that $b$ is even, which corresponds to the case when $n=4k$ for some $k$. We have $a+b=n/2-i \geq 2$ since $i \leq n/2-2$, proving the statement for $n = 4k$.

If $b$ is odd, we need to check that $a+b \geq 3$. It is enough to consider the case $b=1$, and show that this choice implies that $a\geq 2$. Substituting $b$ for $1$ and $a$ for $i$ in the equality $b = n/2 - 2i$ gives that $n/2 -1= 2a$, so  we need to check that  $(n/2-1) \geq 4$. This inequality is satisfied for every even $n \geq 10$, proving the statement for $n=2k \geq 10$, $k$ odd.

The case not covered is $n = 6, \mu-n = 4$. But $n+4 = 10 \notin \Omega'_{6}$, which means that we are done with the case $n$ even.

{\bf The case $n$ odd}.

It is possible to use Theorem \ref{thm:MNS} also in this case, but the argument is more detailed, and it will leave some cases to be checked by  calculations, so we we instead lean on Corollary \ref{cor:main}. However, we  can use  Theorem \ref{thm:MNS} for two of the base cases for $n=7$. Notice that the case $n = 5$ follows from Corollary \ref{cor:inj}.

Let $n = 7$. Consider the interval $[10,12]$.
By Theorem \ref{thm:MNS}, the quotients corresponding to 
$$(x_1,x_2)^2 + (x_3,x_4)^2 + (x_5,x_6)^2 + (x_7)^2, \text{ and } (x_1,x_2,x_3)^2 + (x_4,x_5)^2 + (x_6,x_7)^2$$ both fail the WLP. For the remaining case, $\mu = 11$, we let $$I = (x_1^2,\ldots,x_7^2, x_1 x_7,x_2 x_7, x_3 x_7, x_4 x_7),$$ and check by a computation that $S/I$ fails the WLP.   Together with Corollary \ref{cor:inj}, this proves the statement for $n=7$.

 We assume the statement to hold true for some 
$n \geq 7$, and using Corollary \ref{cor:main}, we conclude that the statement holds true for $n + 2$ variables in the interval $[(n+2)+3, \Delta(n,2) + 2]$.

It rests to show that for each $$
\mu \in \left[\Delta(n,2) + 3, \Delta(n+2,2) \right]
$$
there is an ideal in codimension $n+2$ generated by $\mu$ elements such that the quotient fails the WLP. We now repeat the argument from the proof of Proposition \ref{prop:n+2}; it is enough to show that $2(n+2)-1 \leq \Delta(n,2) +3 $, since  we in this case can use Corollary \ref{cor:inj} to draw our conclusion. But this inequality is easily seen to hold for every $n \geq 5$.

\end{proof}

\section{Ideals with $n+1$ minimal generators and the failure of the WLP} \label{sec:n+1}

The first systematic study of the failure of the WLP for monomial almost complete intersection algebras is due to Migliore, Miro-Roig, and Nagel.

\begin{theorem} \cite[Theorem 4.3]{MMN} \label{thm:MMN}
Let $n \geq 3$. Then the algebra
$$\Bbbk[x_1,\ldots,x_n]/(x_1^n, \ldots,x_n^n,x_1 \cdots x_n)$$ fails the WLP by surjectivity in degree $\binom{n}{2}-1$.
\end{theorem}

\begin{remark}
The proof of \cite[Theorem 4.3]{MMN} follows the second approach, first establishing that surjectivity is expected. Surjectivity is then shown to fail by a technical construction. We observe that the use of the inverse system gives a shorter argument for this part. Indeed, the Vandermonde polynomial
$V=\prod_{i<j} (X_i-X_j)$ is in the inverse system of $I + (\ell)$; that $x_i^n \circ V = 0$ and $\ell \circ V = 0$ is straightforward to check from the factorization, while $x_1 \cdots x_n \circ V = 0$ follows from the determinantal representation, where one can observe that the support of each monomial in $V$ is at most $n-1$. Since $V$ is an element of degree $\binom{n}{2}$, this shows that surjectivity fails in degree 
$\binom{n}{2}-1$.

%It would be interesting to give a proof of \cite[Theorem 4.3]{MMN} using the third approach; showing that injectivity fails by giving an explicit element in the kernel of the map.
\end{remark}

 Theorem \ref{thm:MMN} gives an example of a monomial almost complete intersection for each number of variables $\geq 3$, but for our purpose of proving our main result, we need an example for each number of variables $\geq 3$ \emph{and} every degree $d$ such that $\alpha(n,d)= n+1$ .
 
We split the construction of algebras with $n+1$ generators failing the WLP in three parts. In the first part, we deal with the codimension three case, and here we can use existing results from the literature. In the next two parts we handle the case of $n\geq 4$ generators, one part for each parity of $n$. We proceed in both cases by constructing monomial almost complete intersection algebras in four and five variables for $d \geq 5$ and then use Corollary \ref {cor:main} to extend the results to any $n \geq 4$.

\subsection{The case $n=3$}
Recall that $$\alpha(3,d) = \begin{cases}
4 & \text{if } d \equiv 3 \pmod 6\\
5  & \text{if } d \not\equiv 3 \pmod 6.
\end{cases}$$ Thus we aim to show here the existence of an algebra $S/I$ failing the WLP, where $I$ is generated by $4$ elements of degree $d=6k+3$. It turns out that this result is available in the literature.

\begin{proposition} \cite[Theorem 6.]{GLN} \label{prop:n3macifails}

Let $d=6k+3$. Then any algebra $S/I$, where $I = (x_1^d,x_2^d,x_3^d,x_1^a x_2^b x_3^c)$, with $a+b+c=d$, $4k + 2>a \geq b \geq c$, and where at least two of the numbers $a,b,c$ are equal, fails the WLP.

\end{proposition}

Notice that Theorem \ref{thm:MMN} and Proposition \ref{prop:n3macifails} both covers the Togliatti system $\Bbbk[x_1,x_2,x_3]/(x_1^3,x_2^3,x_3^3,x_1 x_2 x_3)$, which we have already mentioned in Remark \ref{rmk:cor} as a special case.

We remark further that it is conjectured, see \cite[Conjecture 2]{GLN}, that the condition $d=6k+3=a+b+c, 4k + 2>a \geq b \geq c$, and at least two of the numbers $a,b,c$ are equal, describes all 
monomial almost complete intersection algebras $\Bbbk[x_1,x_2,x_3]/(x_1^d,x_2^d,x_3^d,x_1^a x_2^b x_3^c)$ failing the WLP. See also Proposition \ref{prop:n3maciholds} which is relevant in this context.

\subsection{The case $n\geq 4$, $n$ even}
\begin{lemma}\label{lemma:base_induction_even}
Let $I=(x_1^d,x_2^d,x_3^d,x_4^d,x_1^3x_2^{d-3})$, where  
$d\geq 5$. Then $S/I$ fails the WLP by surjectivity in degree $2d-3$.
\end{lemma}
\begin{proof}
Consider the polynomial $f = (X_1-X_2)^{d-1}(X_3-X_4)^{d-1}$ in the inverse system $I^{-1}$ of degree $2d-2$ and observe that $\ell\circ f = 0$. %, for $\ell=x_1+x_2+x_3+x_4$. 
This implies that the differentiation map $\circ \ell : (I^{-1})_{2d-2}\rightarrow (I^{-1})_{2d-3}$ is not injective or equivalently $\times \ell : (S/I)_{2d-3}\rightarrow (S/I)_{2d-2}$ is not surjective. \par  
\noindent It remains to show that $\HF(S/I,2d-3)\geq \HF(S/I,2d-2)$. Define 
$$I_1 :=(x_1^d,x_2^d,x_1^3x_2^{d-3})\subset S_1:=\Bbbk[x_1,x_2],\quad  \text{and}\quad  I_2:=(x_3^d,x_4^d)\subset S_2:=\Bbbk[x_3,x_4].$$
Since $S/I \cong S_1/I_1\otimes S_2/I_2$ for every $t$ we have 
\begin{equation}\label{eq:lem6}
\HF(S/I,t)=\sum_{i=0}^{t}\HF(S_1/I_1,i)\HF(S_2/I_2,t-i).
\end{equation}
We recall that 
\begin{align*}
\HF(S_2/I_2,i) = \begin{cases} i+1 & \text{if}\hspace{2mm}0\leq i\leq d-1,\\
2d-1-i& \text{if}\hspace{2mm}d\leq i\leq 2d-2.
\end{cases}
\end{align*}
Since $I_1$ is generated in degree $d$, $\HF(S_1/I_1,i)= i+1$ for every $0\leq i\leq d-1$. To obtain $\HF(S_1/I_1,i)$ for every $d\leq i\leq 2d-2$ we observe that $\HF(S_1/I_1,i) = \HF(S_1/(x_1^d,x_2^d),i)-\#B_i=\HF(S_2/I_2,i)-\#B_i$ where $B_i=\{x_1^ax_2^{i-a}\mid 3\leq a\leq d-1,\hspace{1mm}\text{and}\hspace{1mm}d-3\leq i-a\leq d-1\}$. It is easy to see that 
\begin{Small}
\begin{align*}
\#B_i = \begin{cases} 1 & \text{if}\hspace{2mm}i = d, 2d-2,\\
2& \text{if}\hspace{2mm} i = d+1,2d-3,\\
3& \text{if}\hspace{2mm} d+2\leq i\leq 2d-4, 
\end{cases},\hspace{2mm}\text{so}\hspace{1mm} \HF(S_1/I_1,i) = \begin{cases} i+1 & \text{if}\hspace{2mm}0\leq i\leq d-1,\\
d-2& \text{if}\hspace{2mm} i = d,\\
d-4& \text{if}\hspace{2mm} i = d+1,\\
2d-4-i& \text{if}\hspace{2mm} d+2\leq i\leq 2d-4,\\
0& \text{if}\hspace{2mm} 2d-3\leq i\leq 2d-2.
\end{cases}
\end{align*}
\end{Small}
Using Equation \ref{eq:lem6} and the Hilbert functions of $S_1/I_1$ and $S_2/I_2$ we get that
\begin{align*}
\HF(S/I,2d-3) &= \sum_{i=0}^{2d-3}\HF(S_2/I_2,i)\HF(S_1/I_1,2d-3-i)\\
 &= \sum_{i=0}^{d-3}(i+1)(i+2) + 2d(d-1)+(d-2)^2+(d-4)(d-3) \\&+ \sum_{d+2}^{2d-4}(2d-4-i)(2d-2-i)
\end{align*}
and 
\begin{align*}
\HF(S/I,2d-2) &= \sum_{i=0}^{2d-2}\HF(S_2/I_2,i)\HF(S_1/I_1,2d-2-i)\\
 &= \sum_{i=0}^{d-2}(i+1)^2 + d^2+(d-2)(d-1)+(d-4)(d-2)\\
 &+ \sum_{d+2}^{2d-4}(2d-4-i)(2d-1-i).
\end{align*}
We conclude that 
\begin{align*}
\HF(S/I,2d-3)-\HF(S/I,2d-2) &=  \sum_{i=0}^{d-2}(i+1)^2 - (d-1)^2 -\sum_{d+2}^{2d-4}(2d-4-i)\\
&+(d-1)(d-2)-(d-4)=2d-9
\end{align*}
and we observe that $2d-9>0$ for every $d\geq 5$.
\end{proof}

\begin{lemma} \label{lemma:macievend3d4}
The monomial almost complete intersection algebra
$$A=\Bbbk[x_1,x_2,x_3,x_4,x_5,x_6]/(x_1^3,x_2^3,x_3^3,x_4^3,x_5^3,x_6^3,x_1 x_2 x_3)$$
fails the WLP in degree 5. 
\end{lemma}

\begin{proof}
We observe that $A$ is isomorphic to the tensor product of the Togliatti system algebra $\Bbbk[x_1,x_2,x_3]/(x_1^3,x_2^3,x_3^3,x_1 x_2 x_3)$ which fails the WLP in degree $2$, and 
the complete intersection algebra $ \Bbbk[x_1,x_2,x_3]/(x_1^3,x_2^3,x_3^3).$ By Lemma \ref{lemma:hfci}, the second algebra has an isolated peak in degree $3$. It follows from Theorem \ref{thm:tensor} that $A$ fails the WLP in degree $5$.

\end{proof}

\begin{proposition}\label{prop:macieven}
Let $n=2m$ be an even integer.
\begin{enumerate}
\item For $d \geq 5$ and $n\geq 4$, let $I = (x_1^d,\dots, x_n^d, x_1^3x_2^{d-3})$. \label{i1} 
\item For $d = 4$ and $n\geq 4$, let $I = (x_1^d,\ldots, x_n^d,x_1 x_2 x_3 x_4)$. \label{i2}
\item For $d = 3$ and $n\geq 6$, let $I = (x_1^d,\ldots, x_n^d,x_1 x_2 x_3)$. \label{i3}
\end{enumerate}
Then, in all cases, the WLP fails in degree $m(d-1) -1$.
\end{proposition}
\begin{proof}
We use induction on $m$. We use Lemma \ref{lemma:base_induction_even} as the base case for (\ref{i1}), Theorem \ref{thm:MMN} (in the special case $n=4$) for (\ref{i2}), Lemma \ref{lemma:macievend3d4} for 
 (\ref{i3}), and Corollary \ref{cor:main} for the induction step.

For the specific degree where the WLP fails, we consider (\ref{i1}) and leave the other two for the reader to verify. By Lemma   \ref{lemma:base_induction_even} the WLP fails in degree 
$2d-3$, so by Corollary \ref{cor:main}, the WLP fails in degree $2d-3 + (m-2)(d-1) = md -1 - m = m(d-1)-1.$

\end{proof}

\subsection{The case $n\geq 5$, $n$ odd}

Although we will give results on the WLP only for the case $n \geq 5$, we will rely on an observation of the Hilbert series of a certain codimension three almost complete intersection.

\begin{lemma}\label{lemma:base_induction_old}
Let $d\geq 5$, $d\neq 6$ and $I=(x_1^d,x_2^d,x_3^d,x_1^{\lfloor\frac{d}{2}\rfloor}x_2^{\lceil\frac{d}{2}\rceil})$. Then $$\HF \left(S/I, \left \lfloor\frac{3d-3}{2} \right \rfloor -1 \right)>\HF \left(S/I,\left \lfloor\frac{3d-3}{2}\right \rfloor \right ).$$
\end{lemma}
\begin{proof}
Denote $t:=\lfloor\frac{3d-3}{2}\rfloor$, so we aim to prove that $\HF(S/I, t -1)>\HF(S/I,t).$ We define the complete intersection ideal $J := (x_1^d,x_2^d, x_3^d)$ and observe that $\HF(S/I,i)=\HF(S/J,i)$ for every $0\leq i\leq d-1$ and $\HF(S/I,i)=\HF(S/J,i)-\HF(S,i-d)$ for every $d\leq i\leq d+\lfloor\frac{d}{2}\rfloor-1$.  Therefore, since $d\leq t \leq d + \lfloor\frac{d}{2}\rfloor-1$
\begin{small}
\begin{equation}\label{eqt}
\HF(S/J,t)-\HF(S/I,t)-\Big(\HF(S/J,t-1)-\HF(S/I,t-1)\Big)=\HF(S,t-d)-\HF(S,t-1-d).
\end{equation}
\end{small}
Notice that for $d=5,6$, $t-d=1$ and  for every $d\geq 7$ we have $t-d\geq 2$.  Thus
\begin{align}\label{eqt2}
\HF(S,t-d)-\HF(S,t-1-d) &= 2\quad  \text{if } d=5,6, \quad \text{and}\\
\HF(S,t-d)-\HF(S,t-1-d) &\geq  3\quad  \text{if } d\geq 7
\end{align}
Now we claim that 
\begin{equation}\label{eqt3}
\HF(S/J,t)-\HF(S/J,t-1)=\begin{cases}
1 & \text{if } d \text{ is odd}, \\
2 & \text{if } d \text{ is even}.\\
\end{cases}
\end{equation}
Let $S^\prime:=\Bbbk[x_2,x_3], J^\prime=(x_2^d,x_3^d)$ and notice that 
\begin{align*}
(S/J)_{t-1}= x_1^{d-1}(S^\prime/I^\prime)_{t-d}\oplus x_1^{d-2}(S^\prime/I^\prime)_{t-d+1}\oplus \cdots \oplus (S^\prime/I^\prime)_{t-1},
\end{align*}
and 
\begin{align*}
(S/J)_{t}= x_1^{d-1}(S^\prime/I^\prime)_{t-d+1}\oplus x_1^{d-2}(S^\prime/I^\prime)_{t-d+2}\oplus \cdots \oplus (S^\prime/I^\prime)_{t}.
\end{align*}
So 
\begin{equation}\label{eqt4}
\HF(S/J,t)-\HF(S/J,t-1)= \HF(S^\prime/I^\prime,t)-\HF(S^\prime/I^\prime,t-d).
\end{equation}
Since $t-d<d$ we have that $\HF(S^\prime/I^\prime,t-d)=t-d+1$. Moreover, since $t>d$, $\mathcal{B}_t=\{x_2^{d-1}x_3^{t-d+1}, x_2^{d-2}x_3^{t-d+2},\dots ,x_2^{t-d+1}x_3^{d-1}\}$ is a basis for $\left(S^\prime/I^\prime\right)_t$ with $\vert \mathcal{B}_t\vert = (d-1)-(t-d+1)+1=2d-t-1$.  Thus using Equation \eqref{eqt4} we get that 
$$
\HF(S/J,t)-\HF(S/J,t-1)= 2d-t-1 - (t-d+1) = 3d-2t-1 
$$
and the claim follows by the definition of $t$.
Now  by rewriting Equation \eqref{eqt} we get that 
\begin{small}
$$
\HF(S/I,t-1)-\HF(S/I,t)=\HF(S,t-d)-\HF(S,t-1-d)-\Big(\HF(S/J,t)-\HF(S/J,t-1)\Big)
$$
\end{small}
thus by Equations  \eqref{eqt2}, and \eqref{eqt3}  we conclude that 
$\HF(S/I,t-1)-\HF(S/I,t)\geq 1$ for every $d\geq 5$ and $d\neq 6$.
\end{proof}
Contrary to the other proofs in this section, the proof of the next lemma is, except for the case $d=6$, in line with the third approach presented in Section \ref{subsec:techniques}; we give a specific degree for which both injectivity and surjectivity fails. We remark that the $d \neq 6$ part of the proof is similar to the proof of Theorem \ref{thm:tensor}.
\begin{lemma}\label{lemma:n=5case}
Let $d\geq 5$, and let $$I=(x_1^d,\dots , x_5^d,x_1^{\left \lceil\frac{d}{2} \right \rceil}x_2^{\left \lfloor\frac{d}{2}\right \rfloor}).$$ Then $S/I$ fails the WLP in degree $t=\lfloor\frac{5d-5}{2}\rfloor-1$.
\end{lemma}
\begin{proof}
Assume first that $d\neq 6$.
We begin by showing that injectivity fails in degree $t$. By Lemma 
\ref{lemma:base_induction_old}, 
the injectivity of the multiplication map by $x_1+x_2+x_3$ on $$\Bbbk[x_1,x_2,x_3]/(x_1^d,x_2^d,x_3^d,x_1^{\lceil\frac{d}{2}\rceil}x_2^{\lfloor\frac{d}{2}\rfloor})$$ 
fails injectivity in degree $$t-(d-1)=\left\lfloor\frac{5d-5}{2}\right\rfloor-1-(d-1)=\left\lfloor\frac{3d-3}{2}\right\rfloor-1.$$ Thus there exists a non-zero polynomial
$$g\in \Bbbk[x_1,x_2,x_3]_{t-(d-1)}$$
such that 
$$(x_1+x_2+x_3)g\in (x_1^d,x_2^d,x_3^d,x_1^{\lceil\frac{d}{2}\rceil}x_2^{\lfloor\frac{d}{2}\rfloor}).$$ Moreover,  since 
$$\HF(\Bbbk[x_4,x_5]/(x_4^d,x_5^d),d-1)>\HF(\Bbbk[x_4,x_5]/(x_4^d,x_5^d),d),$$ there exists a non-zero polynomial $$h\in \Bbbk[x_4,x_5]_{d-1}$$
such that $(x_4+x_5)h\in (x_4^d,x_5^d)$. We let $f=gh\in (S/I)_{t}$ and observe that $$(x_1+\cdots +x_5)f=(x_1+x_2+x_3)gh+(x_4+x_5)gh\in I.$$
Thus the map $\times \ell:(S/I)_{t}\rightarrow(S/I)_{t+1}$ is not injective.

\noindent Now we show that the surjectivity of the multiplication map $\times \ell:(S/I)_{t}\rightarrow(S/I)_{t+1}$ fails as well. 
Let $$F=(X_1-X_2)^{\lfloor\frac{3d-3}{2}\rfloor}(X_3-X_4)^{d-1}\in (I^{-1})_{t+1}.$$
We have that $\ell\circ F=0$, so the map $\circ \ell:(I^{-1})_{t+1}\rightarrow(I^{-1})_{t}$ is not injective which 
completes the first part of the proof.

It remains to prove the statement for $d=6$. A standard inclusion-exclusion computation gives that the value of the Hilbert function in degree $11$ equals 
%$\binom{4+11}{11} - 6 \cdot \binom{4+5}{5}  + 2 \cdot \binom{4+2}{2} = 
$639$, and that the value in degree $12$ equals $642$.
%binomial(4+12,12) - 6*binomial(4+6,6)  + 2* binomial(4+3,3) + 12

We next give four linearly independent elements of degree $12$ in the inverse system of $I + (\ell)$. This shows that the multiplication by $\ell$ map from degree $11$ to $12$ cannot not injective.

Consider the linearly independent elements \begin{align*}
F_1 = (X_3-X_5)(X_1-X_4),
F_2 = (X_3-X_5)(X_2-X_4),\\
F_3 = (X_4-X_5)(X_1-X_3),
F_4 = (X_4-X_5)(X_2-X_3),
\end{align*}
and let 
$$G = 
(X_4-X_5)^2 (X_3-X_5)^2  (X_3-X_4)^2  (X_1-X_2)^4.$$ It is straightforward to check that 
$F_1 G, F_2 G, F_3 G, F_4 G$ belongs to the inverse system. Since $\lfloor\frac{5 \cdot 6-5}{2}\rfloor-1 = 11$, the proof is now complete.

\end{proof}

As for the even case, we need a lemma to cover the situations $d=3$ and $d= 4$.
\begin{lemma} \label{lemma:macioddnd3d4}
The monomial almost complete intersection algebra 
$$\Bbbk[x_1,x_2,x_3,x_4,x_5]/(x_1^4,x_2^4,x_3^4,x_4^4,x_5^4,x_1 x_2 x_3 x_4)  $$
 fails the WLP  by injectivity in degree 6.
 \end{lemma}

\begin{proof}
We proceed similarly as in the proof of the $d=6$-case of Lemma \ref{lemma:n=5case}. We determine first the value of the Hilbert function in degree $6$ and $7$ to be $120$ and $124$, respectively.
%using standard inclusion-exclusion. The value in degree $6$ equals $\binom{4+6}{6} - \binom{4+2}{2} \cdot 6 = 120$ and the value in degree $7$ equals $\binom{4+7}{7} - \binom{4+3}{3} \cdot 6  + 4 = 124.$ 
Thus the dimension in degree $7$ is expected to equal $4$. 

We proceed by giving five linearly independent elements of degree seven in the inverse system of $(x_1^4,x_2^4,x_3^4,x_4^4,x_5^4,x_1 x_2 x_3 x_4,x_1+x_2+x_3+x_4+x_5),$ which shows that injectivity fails.

Let \begin{align*}
f  =(X_4-X_5)^2  (X_2-X_3)  (X_1-X_3) \cdot (X_1-X_2) \cdot
((X_5-X_3)(X_4-X_2) +\\ (X_5-X_2)(X_4-X_1)+(X_5-X_1)(X_4-X_3))
\end{align*} and let 

$$g = (X_1-X_2)(X_3-X_4)(   h_1-h_2),$$
where

\begin{align*}
h_1 = (X_2-X_5)(X_1-X_5)((X_4-X_5)^2(3X_1+3X_2-4X_4-2X_5) +\\ (X_3-X_4)(X_3+2X_4-3X_5)(X_1+X_2-2X_4))
\end{align*}
and 
\begin{align*}
h_2=(X_4-X_5)(X_3-X_5)(2( 2(X_2-X_5)^2 + (X_1-X_2)(X_1+X_2-2X_5))(X_3-X_5)+\\
(X_4-X_5)( 2 (X_2-X_5)(2X_2-3X_3+X_5) + (X_1-X_2)(2X_1+2X_2-3X_3-X_5))).
\end{align*}

To verify that $f$ is in the inverse system can be done without too much effort, but to do so  for $g$ is computationally harder, so we verify the latter statement with a computation in Macaulay2.
Moreover, we claim that the elements $f, (14) \circ f, (24) \circ f, g, (23) \circ g$ are linearly independent, where the operation $\circ (ij)$ means interchanging $x_i$ and $x_j$. This can be checked by noticing that $$X_1^2X_2^3X_5^2, X_1X_3X_4^2X_5^3, X_1^2X_4^3X_5^2, X_1^2X_3^2X_5^3,X_1^2X_2^2X_5^3$$ occurs uniquely in  $f, (14) \circ f, (24) \circ f, g$, and $ (23) \circ g$, respectively.

\end{proof}

\begin{proposition} \label{prop:maciodd}
Let $n= 2m+1$ be an odd integer.
\begin{enumerate} 
\item Let $d \geq 5$ and suppose $n \geq 5$. Let $I=(x_1^d,\dots , x_n^d,x_1^{\lceil\frac{d}{2}\rceil}x_2^{\lfloor\frac{d}{2}\rfloor})$. Then $S/I$ fails the WLP 
in degree $\lfloor\frac{n(d-1)}{2}\rfloor-1$.  \label{ii1}

\item Let $d = 4$ and suppose $n \geq 5$. Let $I = (x_1^d,\ldots, x_n^d,x_1 x_2 x_3 x_4)$. Then $S/I$ fails the WLP in degree $3m.$  \label{ii2}

\item Let $d = 3$ and suppose $n \geq 3$. Let $I = (x_1^d,\ldots, x_n^d,x_1 x_2 x_3)$. Then $S/I$ fails the WLP in degree $2m.$  \label{ii3}

\end{enumerate}
\end{proposition}

\begin{proof}
As in Proposition \ref{prop:macieven}, we use induction on $m$. We use Lemma \ref{lemma:n=5case} as the base case for (\ref{ii1}), Lemma \ref{lemma:macioddnd3d4}
as the base case for (\ref{ii2}), and the Togliatti system (the codimension three case of Theorem \ref{thm:MMN}) for (\ref{ii3}), and Corollary \ref{cor:main} for the induction step. 

For the specific degree where the WLP fails, we consider (\ref{ii1}) and leave the other two for the reader to verify. By Lemma  \ref{lemma:n=5case} the WLP fails in 
degree $\lfloor\frac{5d-5}{2}\rfloor-1$, so by Corollary \ref{cor:main}, the WLP fails in degree 
\begin{align*}
\left \lfloor\frac{5d-5}{2} \right \rfloor-1 + (d-1)(m-2) = 
\left \lfloor \frac{5d-5 + 2 (d-1)(m-2)}{2} \right \rfloor -1 = \\
%\lfloor \frac{(d-1)(5+2(m-2))}{2} \rfloor -1 =  \\
\left \lfloor \frac{(d-1)(2m+1)}{2} \right \rfloor -1 =\left \lfloor \frac{n(d-1)}{2} \right \rfloor -1.
\end{align*}

\end{proof}

\section{Presence of the WLP} \label{sec:six}

We now shift perspective to the existence of ideals $I$ such that the corresponding quotient $S/I$ has the WLP. In the first part we collect results for the codimension three case, and in the second part, we deal with the remaining cases.

\subsection{The case $n=3$}

\begin{lemma}
Let $n=3$ and $d=2$. Then every artinian algebra of the form $S/I$, where $I$ is monomial ideal, has the WLP. 
\end{lemma}

\begin{proof}
We can use Theorem \ref{thm:MNS} for all cases except $I=(x_1^2,x_2^2,x_3^2,x_1 x_2, x_1 x_3)$, for which the corresponding quotient can be checked by hand to satisfy the WLP. 
\end{proof}

\begin{proposition} \cite[Proposition 1.]{GLN} \label{prop:n3maciholds}
Let $n=3$ and let $d$ be an integer not on the form  $6k+3$. Let $I$ be an ideal minimally generated by $n+1$ monomials of degree $d$. Then $S/I$ has the WLP.
\end{proposition}

\subsection{The case $n\geq 4$}

\begin{lemma} \label{lemma:n+1,2}
Let $n \geq 4$ and let $d =2$. Let $I$ be an artinian ideal minimally generated by $n+1$ quadratic monomials. Then $S/I$ has the WLP.
\end{lemma}
We give two short proofs. Notice first that up to isomorphism, there is only one monomial almost complete intersection for $d=2$; $I= (x_1^2,\ldots,x_n^2,x_1 x_2)$.
\begin{proof}[First proof]
Use Theorem \ref{thm:MNS} with $d_1=2,d_2=\ldots =d_{n-1} = 1.$
\end{proof}
\begin{proof}[Second proof]
By \cite[Theorem 5]{GLN}, the algebra $S/(x_1^d,\ldots,x_n^d,x_1^{d-1} x_2)$ has the SLP when $n \geq 3, d\geq 2$. %, so $S/I$ has the WLP.
\end{proof}

To tackle $n+2$ generators, we need to recalling the following result.

\begin{lemma}Special case of  \cite[Theorem 1] {GLN}\label{lemma:glue}.
Let $K$ be a monomial ideal, let $$I = K + (x_j), \text{ and let } J = K : (x_j)$$ for some variable $x_j$. 
Assume that $S/I$ and $S/J$ both have the WLP and that

\[\HF(S/I,i)\!<\!\HF(S/I,i+1) \implies \HF(S/J,i-1)\! \le\! \HF(S/J,i)\]
and
\[\HF(S/I,i)\!>\!\HF(S/I,i+1) \implies \HF(S/J,i-1)\! \ge \! \HF(S/J,i)\]
holds for all i. Then $S/K$ has the WLP.
\end{lemma}

\begin{proposition}  \label{prop:n,n+2,2}
Let $n \geq 5$ be odd and consider an artinian ideal minimally generated by $n+2$ quadratic monomials. Then the corresponding quotient has the WLP.
\end{proposition}
\begin{proof}
Up to ismorphism, there are two ideals to consider; $(x_1^2,\ldots,x_n^2,x_1 x_2, x_3 x_4)$ and $(x_1^2,\ldots,x_n^2,x_1 x_2, x_1 x_3)$. 

The argument for the first ideal follows directly by letting $d_1=2,d_2=2, d_3=\ldots=d_{n-2} = 1$ in Theorem \ref{thm:MNS}.

We now turn to the second ideal. Let  $$K=(x_1^2,\ldots,x_n^2,x_1 x_2, x_1 x_3).$$ To prove that $S/K$ has the WLP, let $$I=K+(x_1) = (x_1,x_2^2,\ldots,x_n^2)$$ and let $$J=K:(x_1) = (x_1,x_2,x_3,x_4^2,\ldots,x_n^2),$$ which are complete intersections, and thus the corresponding quotients both
have the WLP. 
By Lemma \ref{lemma:hfci}, the Hilbert series of $S/I$ has a isolated peak in degree $(n-1)/2$ %and $S/I$ has socle degree $n-1$, 
while the Hilbert series of $S/J$ has a isolated peak in degree $(n-3)/2 = (n-1)/2 -1$. %and socle degree $n-3$.

To apply Lemma \ref{lemma:glue} we now verify that the two conditions in the lemma are satisfied. 

Suppose first that $i \leq (n-3)/2$. Then $$\HF(S/I,i)\!<\!\HF(S/I,i+1), \text{ since }  i+1\leq(n-1)/2$$ and $$\HF(S/J,i-1)\! < \! \HF(S/J,i), \text{ since } i \leq (n-3)/2.$$

Suppose instead that $i \geq (n-3)/2+1$. Then $$\HF(S/I,i)\!>\!\HF(S/I,i+1) \text{, since } i \geq (n-1)/2$$ and $$\HF(S/J,i-1)\! > \! \HF(S/J,i)  \text{, since } i -1 \geq (n-3)/2.$$ 
It follows %by Lemma \ref{lemma:glue} 
that $S/K$ has the WLP.

\end{proof}

We are left with some special cases which we verify by calculations. 

\begin{lemma} \label{lemma:remainingcases}
Let $(n,d,\mu) \in \{(4,3,5), (5,2,8), (6,2,10)\}.$ If $I \subset \Bbbk[x_1,\ldots,x_n]$ is an artinian monomial ideal minimally generated by 
$\mu$ elements of degree $d$, then $\Bbbk[x_1,\ldots,x_n]/I$ has the WLP.
\end{lemma}

\begin{proof}
For the case $(4,3,5)$ there are two cases to consider up to isomorphism. For the case $(5,2,8)$, there are four undirected graphs on five vertices with three edges, each corresponding naturally to an artinian ideal with $8$ generators. In the case $(6,2,10)$, there are eight undirected graphs on six vertices with four edges, each corresponding naturally to an artinian ideal with $10$ generators.

In all cases it can be checked by Macaulay2 (or even by hand) that the corresponding quotients all have the WLP.
\end{proof}

\section{Proof of the main result} \label{sec:seven}
\setcounter{subsection}{1}
We now have everything we need to prove Theorem \ref{thm:main}.
\begin{proof}[Proof of Theorem \ref{thm:main}]
We split the proof in two parts.

\hspace*{0.5 cm}

{\bf The case $n=3.$} Recall that $$\alpha(3,d) = \begin{cases}
4 & \text{if } d \equiv 3 \pmod 6\\
5  & \text{if }  d \not\equiv 3 \pmod 6
\end{cases}
$$ and that 
$ \Sigma_{3,d} = [\alpha(3,d),\beta(3,d)].$

By Corollary \ref{cor:inj} we have established that for every $\mu$ in the interval $[5,\Delta(3,d)]$, there is a monomial ideal minimally generated by $\mu$ elements failing the WLP by injectivity. By Corollary \ref{cor:surj} we have that for every $\mu$ in the interval $[\Delta(3,d),\beta(3,d)]$ there is a monomial ideal minimally generated by $\mu$ elements failing the WLP by surjectivity.
Moreover, when $d \geq 4$, the set $[5,\Delta(3,d)]$ is non-empty by Corollary \ref{cor:inj}. It is straightforward to verify that $5 \leq \beta(3,d)$ holds in this interval, so we conclude that the interval $[5,\beta(3,d)]$ is non-empty. 

By Proposition \ref{prop:n3macifails} there is a monomial ideal minimally generated by $4$ elements failing the WLP when $d = 6k+3$.

We are left with the sharpness part. By Theorem \ref{thm:ALthm}, the upper bound is sharp. By Proposition \ref{prop:n3maciholds}, each monomial ideal minimally generated by $n+1$ elements has the WLP when $d$ cannot be written on the form $6k+3$. When  $\mu = n$ the result follows by Stanley's result on monomial complete intersections. This concludes the proof, since there are no artinian algebras whose corresponding ideals are generated by $\mu < n$ elements.

\hspace{0.5 cm}

{\bf The case $n \geq 4.$}
We begin by showing that for each  $\mu \in \Sigma_{n,d}$ there
exists an ideal minimally generated by $\mu$ elements whose corresponding algebra fails the WLP. In the case $n \geq 4$ we have
$$\alpha(n,d)=\begin{cases}
n+2 & \text{if } n \geq 4 \text{ is even}, d = 2 \\
9  & \text{if } n = 5, d = 2\\
n+3  & \text{if } n \geq 7 \text{ is odd}, d= 2\\
6 & \text{if } n = 4, d=3\\
n+1  & \text{otherwise (if } n \geq 5, d \geq 3 \text{ or } n \geq 4, d \geq 4 \text{)}\\
\end{cases}$$
and 
$$\Sigma_{n,d}=
\begin{cases}
   [\alpha(n,d),\beta(n,d)] \setminus \{10\} = \{8,9,11,12,\ldots,17\}&  \text{ if } (n,d) = (6,2)\\
 [\alpha(n,d),\beta(n,d)] & \text{ otherwise.}\\
 \end{cases}
 $$

    We consider first the case $d=2$. By Proposition \ref{prop:d=2}, the statement holds true for every $\mu$ in the interval $[\alpha(n,2), \Delta(n,2)]$ when 
    $n \neq 6$ and for every $\mu$ in the interval $[\alpha(n,2), \Delta(n,2)] \setminus \{10\}$ when $n = 6$. By Corollary \ref{cor:surj} the statement holds true for every $\mu$ in the interval 
    $[\Delta(n,2), \beta(n,2)]$.
    
    We next consider the case $d \geq 3$.  By Proposition \ref{prop:n+2},
    the statement holds true for every $\mu$ in the interval
    $[n+2, \Delta(n,d)]$. By Corollary \ref{cor:surj} the statement holds true for every $\mu$ in the interval 
    $[\Delta(n,d), \beta(n,d)]$. By Proposition \ref{prop:macieven}, the statement holds true for $n+1$ generators when $d =3$, $n$ is even, and $n \geq 6$. By the same proposition, the statement holds true for $n+1$ generators when $d\geq 4$, $n$ is even, and $n \geq 4$.
By  Proposition \ref{prop:maciodd}, the statement holds true for $n+1$ generators when $d \geq 3$, $n$ is odd and $n \geq 5$.

We will now prove the sharpness part, that is, we will show that for every ideal $I$ minimally generated by $\mu \notin \Sigma_{n,d}$ elements, the algebra $S/I$ has the WLP. For this part we rely on the results in Section \ref{sec:six}. 

Notice first that if $\mu \leq n-1,$ then the corresponding algebra cannot be artinian, and if $\mu = n$, then for the corresponding algebra to be artinian, it has to be a complete intersection, which we know have the WLP. Moreover, by Theorem \ref{thm:ALthm}, we know that for every ideal minimally generated by $\mu > \beta(n,d)$ monomials of degree $d$, the corresponding algebra has the WLP. 

Thus it is enough to consider the cases $n+1 \leq \mu < \alpha(n,d)$ together with the special case $\mu = 10$ when $n = 6, d=2$. 

Again we begin by assuming that $d=2$. By Lemma \ref{lemma:n+1,2}, every monomial almost complete intersection in $n \geq 4$ variables has the WLP, and by Proposition \ref{prop:n,n+2,2} every ideal generated by $n+2$ elements corresponds to an algebra that has the WLP when $n\geq 5$ is odd. We are left with the cases $n=5, \mu=8$ and $n=6, \mu = 10$, which both follows by Lemma \ref{lemma:remainingcases}.

We now assume that $d \geq 3$. In this case we only need to show that every cubic monomial almost complete intersection in four variables has the WLP, which follows from Lemma \ref{lemma:remainingcases}.

Finally, we need to prove that $\Sigma_{n,d}$ is non-empty. By Corollary \ref{cor:surj}, the interval  $[\Delta(n,d), \beta(n,d)]$ is non-empty, so it is enough to verify that $\alpha(n,d) \leq \Delta(n,d)$, which is straightforward to do. 

\end{proof}

We emphasize the constructive nature of  our results by means of an example.

\begin{example} 

Let $n = 8, d=4,\mu = 13.$ We have $\mu\in \Sigma_{8,4}=\left[\alpha(8,4),\beta(8,4)\right]=\left[9,322\right]$, so by Theorem \ref{thm:main}, there exists an ideal $I$ minimally generated by $13$ elements so that $S/I$ fails the WLP. Following the proof of Theorem \ref{thm:main} we construct such an ideal. Since  $\mu\in \left[n+2,\Delta(8,4)\right] = \left[10, 210 \right]$, the proof of Theorem \ref{thm:main} in this case relies on Proposition \ref{prop:n+2}.

We now follow the proof of Proposition \ref{prop:n+2} for the construction. Induction on $n$ is used in the proof, where the base case is an ideal in $n-4=4$ variables and $\mu-4=9$ generators failing the WLP. For the induction process, Corollary \ref{cor:main} is used, so the number of variables and generators in each step are increased by two.

We have $\mu - 4=9\in \left[2n-1,\binom{4+4-1}{4} - \binom{4+3-1}{3}\right]=\left[7,15\right]$, so by Corollary \ref{cor:inj} any ideal in $\Bbbk[x_1,x_2,x_3,x_4]$ with $9$ generators containing 
$$(x_1^4,x_2^4,x_3^4,x_4^4,x_1 x_4^3, x_2 x_4^3, x_3 x_4^3)$$ 
yields an algebra failing the WLP by injectivity in degree $3$.

Let us choose 
$$(x_1^4,x_2^4,x_3^4,x_4^4,x_1 x_4^3, x_2 x_4^3, x_3 x_4^3,x_1^3 x_2,x_1^2 x_2^2).$$ 
This establishes the base case of the induction. We now use Corollary \ref{cor:main} twice to obtain the algebra 
$$\Bbbk[x_1,\ldots,x_8]/(x_1^4,x_2^4,x_3^4,x_4^4,x_1 x_4^3, x_2 x_4^3, x_3 x_4^3,x_1^3 x_2,x_1^2 x_2^2, x_5^4,x_6^4,x_7^4,x_8^4)$$
which fails the WLP in degree $3+ 2 \cdot (4-1) = 9.$

\end{example}

\section{A final remark} \label{sec:rmk}
\setcounter{subsection}{1}
In some of the base cases in our induction we have information on why the WLP fails; in which degree it happens, and if it is due to injectivity or surjectivity. But when we use Corollary \ref{cor:main} for the induction step, we lose some of this data. To be precise, if $A$ fails the WLP by injectivity or  surjectivity in degree $a$, then, by Corollary \ref{cor:main}, the algebra $A \otimes \Bbbk[y_1,y_2]/(y_1^d,y_2^d)$ fails the WLP in degree $a + (d-1)$, but we do not get information on whether it fails due to injectivity or surjectivity.

One could hope for a stronger version of Therorem \ref{thm:tensor} and Corollary \ref{cor:main} giving information on why the WLP fails in degree $a + (d-1)$ for the tensored algebra, for instance, it would be pleasant to be able to say that failure due to injectivity (or surjectivity) is preserved when we extend the number of variables using Corollary \ref{cor:main}. However, the following example shows that this is not the case in general. 

\begin{example}
Let $I = (x_1^5,x_2^5,x_3^5,x_1^3 x_2 x_3,x_1^3 x_2^2,x_1^4 x_2,x_1^4 x_3) \subset \Bbbk[x_1,x_2,x_3]$. The Hilbert series of the algebra $A = \Bbbk[x_1,x_2,x_3]/I$ equals $$1+3\,T+6\,T^{2}+10\,T^{3}+15\,T^{4}+14\,T^{5}+13\,T^{6}+10\,T^{7}+6\,T
       ^{8}+3\,T^{9}+T^{10},$$ and the WLP fails by surjectivity in degree $4$, since both $x_1^4$ and $x_1^3 x_2$ are in the kernel of the multiplication by $x_1 + x_2 + x_3$ map. 
       
       Let $B = \Bbbk[y_1,y_2]/(y_1^5,y_2^5)$. By Corollary \ref{cor:main}, the algebra $C:=A \otimes B$ fails the WLP in degree $4+(5-1) = 8.$ But the Hilbert series of $C$ equals $$1+5\,T+15\,T^{2}+35\,T^{3}+70\,T^{4}+117\,T^{5}+171\,T^{6}+223\,T^{7}+
       261\,T^{8}+272\,T^{9}+257\, T^{10} + \cdots$$
       %\,+219\,T^{11}+167\,T^{12}+113\,T^{13}+68
       %\,T^{14}+35\,T^{15}+15\,T^{16}+5\,T^{17}+T^{18}$, 
       so 
       the failure in degree $8$ cannot be due to surjectivity; it is due to injectivity.
       
       Notice that we can use Remark \ref{rmk:cor} to give information on the value of the Hilbert series of $C/(x_1+x_2+x_3+y_1+y_2)$, and the kernel of the multiplication by $\ell$ map from degree $8$ to $9$. Indeed, there are at least two elements in the kernel of the multiplication by $x_1+x_2+x_3+y_1+y_2$ map from degree $8$ to degree $9$ of $C$, and they are given explicitly by $f_1 g_1$ and $f_2 g_1$, where $f_1 =x_1^4, f_2= x_1^3 x_2$ and $g_1= (y_1^4 - y_1^3 y_2 + y_1^2 y_2^2 - y_1 y_2^3 + y_2^4)$, where the latter is in the kernel of the multiplication by $y_1+y_2$ map from degree $4$ to degree $5$ on $B$. A computation in Macaulay2 shows that these two elements in fact spans the kernel; the Hilbert series of $C/(x_1+x_2+x_3+y_1+y_2)$ equals  $$1+4\,T+10\,T^{2}+20\,T^{3}+35\,T^{4}+47\,T^{5}+54\,T^{6}+52\,T^{7}+38\,
       T^{8}+13\,T^{9},$$ 
       where the coefficient $13$ of $T^9$ satisfies $13 = 272-261 + 2$.
\end{example}

\section*{Acknowledgements}
Many of the ideas that we have formalized to proofs has come out from computer experiments in Macaulay2. The first author is supported by the Swedish Research Council grant VR2021-00472. Finally, we thank the anonymous referee for comments which have helped us to improve the presentation of the paper.

\end{document}